\documentclass[10pt,leqno]{amsart}
\usepackage{graphicx}
\usepackage{indentfirst,csquotes}
\usepackage{mathtools}

\usepackage{amssymb,amsthm,amsmath}
\usepackage{xcolor,paralist,hyperref,titlesec,etoolbox}

\newtheorem{theorem}{Theorem}[section]
\newtheorem*{theorem*}{Theorem}
\newtheorem{lemma}[theorem]{Lemma}
\newtheorem*{corollary*}{Corollary}
\newtheorem{corollary}[theorem]{Corollary}
\newtheorem*{proposition*}{Proposition}
\newtheorem{proposition}[theorem]{Proposition}
\newtheorem*{acknowledgments*}{Acknowledgments}

\theoremstyle{definition}
\newtheorem{definition}[theorem]{Definition}

\theoremstyle{remark}
\newtheorem{remark}[theorem]{Remark}

\numberwithin{equation}{section}

\def\zz{\mathbb{Z}}
\def\qq{\mathbb{Q}}

\def\gcd{{\rm gcd}}
\def\lcm{{\rm lcm}}

\def\re{{\rm Re}}

\def\ord{{\rm ord}}

\newcommand{\res}{\operatorname{Res}}

\titleformat{\section}
{\normalfont\Large\bfseries\centering}
{\thesection.}
{1em}
{}

\hypersetup{ colorlinks=true, linkcolor=black, filecolor=black, urlcolor=black }

\usepackage{lipsum}

\begin{document}
\title{Resultants of dynatomic polynomials of $x^d$} %%%%%%%%%%%%
\author{Chih-Chiang Kao}
\date{}
\address{Department of Mathematics, Institute of Science Tokyo, 2-12-1 Ookayama, Meguro-ku, 152-8550, Tokyo, Japan}
\email{kao.c.aa@m.titech.ac.jp}
\maketitle

\let\thefootnote\relax
\footnotetext{MSC2020: Primary 37P05, Secondary 13P15, 37P35.} %%%%%%%%%%
\footnotetext{Key words and phrases. Resultant, dynatomic polynomial, cyclotomic polynomial.}

\begin{abstract}
Let $K$ be a field of characteristic zero, and let $\phi(x)\in K[x]$ be a polynomial of degree at least 2. Denote the $n$-th iterate of $\phi$ by $\phi^n$. The $n$-th dynatomic polynomial of $\phi$ is defined by $$\Phi_{\phi,n}(x) \coloneqq \prod_{k\mid n}(\phi^k(x)-x)^{\mu(n/k)}.$$
In this paper, we specialize to the case $\phi(x) = x^d$. We first establish several properties of $\Phi_{\phi,n}$ that are analogous to those of cyclotomic polynomials. We then combine these properties with known results on resultants of cyclotomic polynomials to determine the resultants of dynatomic polynomials. In particular, for $1\leq n< m$, we show that $\res(\Phi_{\phi,n},\Phi_{\phi,m}) = 1$
if and only if $n\nmid m$. When $n\mid m$, we obtain the explicit formula
$$\res(\Phi_{\phi,n},\Phi_{\phi,m}) = \Phi_{m/n}(d^n)^{\nu_d(n)},$$
where $\Phi_j$ is the $j$-th cyclotomic polynomial, $\nu_d(1) = d-1$, and $\nu_d(n) = \deg \Phi_{\phi,n}$ for $n>1$.
\end{abstract} %%%%%%%%%

\bigskip

\section{Introduction}

Let $K$ be a field of characteristic zero. Let $\phi(x)\in K[x]$ be a polynomial with $\deg\phi\geq 2$. For each positive integer $k$, let $\phi^k(x)$ denote the $k$-th iterate of $\phi$. We also set $\phi^0(x) = x$. We begin by defining the dynatomic polynomial associated with $\phi$.

\begin{definition}
    The $n$-th \textit{dynatomic polynomial} $\Phi_{\phi,n}(x)\in K[x]$ associated with $\phi$ is defined by
    $$\Phi_{\phi,n}(x) = \prod_{k\mid n}(\phi^k(x)-x)^{\mu(n/k)},$$
    where $\mu(\cdot)$ denotes the Möbius function.
\end{definition}

Although the above definition yields a rational function, since the Möbius function may take negative values, Morton and Patel \cite[Theorem 2.5]{Morton94} proved that it is in fact a polynomial over $K$. For more details on dynatomic polynomials, we refer the reader to \cite[Section 4.1]{Silverman07}.

Let $\overline{K}$ be a fixed algebraic closure of $K$. An element $\alpha\in\overline{K}$ is called a periodic point of $\phi$ of period $n$ if $\phi^n(\alpha) = \alpha$. That is, $\alpha$ is a root of the polynomial $\phi^n(x)-x$. We say that $\alpha$ has exact period $n$ if $\phi^n(\alpha) = \alpha$ and $\phi^i(\alpha)\neq \alpha$ for all $1\leq i< n$. We also say that $\alpha$ has formal period $n$ if $\Phi_{\phi,n}(\alpha) = 0$.

It follows immediately from the definitions that
$$\alpha\text{ has an exact period }n \Rightarrow \alpha\text{ has a formal period }n \Rightarrow \alpha\text{ has period }n.$$
On the other hand, the example $\phi(x) = x^2-\frac{3}{4}$ in $\qq[x]$ shows that neither converse implication holds. Indeed, 
$$\Phi_{\phi,1}(x) = \left(x-\frac{3}{2}\right)\left(x+\frac{1}{2}\right),\qquad \Phi_{\phi,2}(x) = \left(x+\frac{1}{2}\right)^2.$$
Thus $-1/2$ has formal period 2 but exact period 1, whereas $3/2$ has period 2 but does not have formal period 2.

The $n$-th cyclotomic polynomial $\Phi_n(x)$, whose roots are precisely the primitive $n$-th roots of unity, is defined by
$$\Phi_n(x) = \prod_{k\mid n}(x^k-1)^{\mu(n/k)}.$$

We recall the multiplier polynomial construction in the case where $K$ is a number field and $\phi^r(x)-x$ is squarefree for every $r\geq 1$. Let $\alpha$ be a root of the dynatomic polynomial $\Phi_{\phi,n}(x)$. The multiplier $\omega_n(\alpha)$ is defined
as the derivative of $\phi^n$ at $\alpha$.  The construction in \cite{Vivaldi92} yields a monic polynomial $\delta_n(x)\in K[x]$ satisfying
$$(\delta_n(x))^n = \prod_{\Phi_{\phi,n}(\alpha)=0}(x-\omega_n(\alpha)),$$
where the product is over the distinct roots of
$\Phi_{\phi,n}$. This is called the $n$-th multiplier polynomial of $\phi$. In this setting, for monic $\phi$, Morton and Vivaldi \cite[Theorem 2.2]{Morton95} relate dynatomic resultants to multiplier polynomials: for any divisor $m$ of $n$ with $m<n$, 
$$\res(\Phi_{\phi,n},\Phi_{\phi,m}) = \pm\Delta_{n,m}^m,$$
where $\Delta_{n,m} = \res(\Phi_{n/m}(x),\delta_m(x))$. Thus, the resultant of $\Phi_{\phi,n}$ and $\Phi_{\phi,m}$ can be expressed in terms of the resultant of a cyclotomic polynomial and the multiplier polynomial, although determining the multiplier polynomials explicitly can be difficult in general.

In this paper, we restrict our attention to the case $\phi(x) = x^d$ and use known results on resultants of cyclotomic polynomials to give a direct arithmetic proof of the following theorem. For convenience, we write $\Phi_{d,n} = \Phi_{\phi,n}$ when $\phi(x) = x^d$.
\begin{theorem*}
    Let $d\geq 2$ and $1\leq n<m$. Then
    \begin{enumerate}
        \item $\res(\Phi_{d,1},\Phi_{d,m}) = \Phi_m(d)^{d-1}$,
        \item For $1<n<m$,
        $$\res(\Phi_{d,n},\Phi_{d,m}) = \begin{cases}
            1,&\text{ if }n\nmid m,\\
            \Phi_{m/n}(d^n)^{\deg \Phi_{d,n}},&\text{ if }n\mid m.
        \end{cases}$$ 
    \end{enumerate}
    In particular, for $1\leq n<m$, the resultant is 1 if and only if $n\nmid m$.
\end{theorem*}

For $\phi(x) = x^d$, if $n\mid m$, the formula is also consistent with the multiplier polynomial relation above; see Remark 4.6. Our proof computes its prime valuations directly from cyclotomic factors and multiplicative orders.

In the next section, we review some well-known results on resultants that will be used throughout the remainder of the paper. In Section 3, we establish some properties of $\Phi_{d,n}$ analogous to those of cyclotomic polynomials. In the final section, we determine the resultant $\res(\Phi_{d,1},\Phi_{d,m})$ for $m>1$ and then compute $\res(\Phi_{d,n},\Phi_{d,m})$ for $1< n< m$.

\begin{acknowledgments*}
    The author expresses sincere gratitude to Professor Yuichiro Taguchi for his continued help and suggestions.
\end{acknowledgments*}

\section{Facts about resultants}

In this section, we briefly review the definition and some well-known properties of resultants that will be used in the following two sections. For further details, we refer the reader to \cite[Chapter 12]{Gelfand94} or \cite[\textup{Chapter 5}]{Waerden70}.

Given any two nonzero polynomials 
$$A = a_nx^n+\cdots+a_1x+a_0\quad\text{ and }\quad B = b_mx^m+\cdots+b_1x+b_0,$$
we define the resultant $\res(A,B)$ as the determinant of their Sylvester matrix:
$$\res(A,B) = \begin{array}{c@{\hspace{0.2em}}l}
    \begin{vmatrix}
        a_n & a_{n-1} & \cdots  & a_1    & a_0     &        &        &     \\
            & a_n     & a_{n-1} & \cdots & a_1     & a_0    &        &     \\
            &         & \ddots  & \ddots & \ddots  & \ddots & \ddots &     \\
            &         &         & a_n    & a_{n-1} & \cdots & a_1    & a_0 \\
        b_m & b_{m-1} & \cdots  & b_1    & b_0     &        &        &     \\
            & b_m     & b_{m-1} & \cdots & b_1     & b_0    &        &     \\
            &         & \ddots  & \ddots & \ddots  & \ddots & \ddots &     \\
            &         &         & b_m    & b_{m-1} & \cdots & b_1    & b_0
    \end{vmatrix} & 
    \begin{array}{l}
        \left. \rule{0pt}{7ex} \right\} m \text{ rows} \\[3.5ex]
        \left. \rule{0pt}{7ex} \right\} n \text{ rows}
    \end{array}
\end{array},$$
where the remaining entries are all zero.

\begin{proposition} 
    Suppose that
    $$A = a_n(x-\alpha_1)\cdots(x-\alpha_n),\qquad B = b_m(x-\beta_1)\cdots(x-\beta_m).$$
    Then:
    \begin{enumerate}
        \item $\res(A,B) = a_n^m\prod_{i = 1}^nB(\alpha_i) = (-1)^{mn}b_m^n\prod_{j = 1}^mA(\beta_j).$
        \item $\res(B,A) = (-1)^{mn}\res(A,B)$.
        \item If $C$ is another nonzero polynomial, then $\res(AB,C) = \res(A,C)\res(B,C)$.
    \end{enumerate}
\end{proposition}
In particular, the resultant vanishes if and only if the two polynomials have a common root over the algebraic closure. Furthermore, applying Proposition 2.1(2) and (3), we can obtain $\res(A,BC) = \res(A,B)\res(A,C)$.

To compute the resultants of dynatomic polynomials of $x^d$, it is helpful to know the resultants of cyclotomic polynomials. Lehmer \cite{Lehmer30}, Diederichsen \cite{Diederichsen40}, Apostol \cite{Apostol70}, Louboutin \cite{Louboutin97}, and Dresden \cite{Dresden12} have provided different proofs of the following formula. We omit the proofs and state the result below.

\begin{theorem}
    For integers $0<n<m$, the resultant of $\Phi_n$ and $\Phi_m$ is
    $$\res(\Phi_n,\Phi_m) = \begin{cases}
        p^{\varphi(n)} &\text{ if }m/n\text{ is a positive power of a prime }p,\\
        1 &\text{otherwise},
    \end{cases}$$
    where $\varphi(\cdot)$ denotes Euler's totient function.
\end{theorem}

Note that $\deg\Phi_n = \varphi(n)$ is even for $n\geq 3$. Hence, for integers $0<m<n$ with $n\geq 3$, Proposition 2.1(2) gives $\res(\Phi_m,\Phi_n) = \res(\Phi_n,\Phi_m)$. On the other hand, 
$$\res(\Phi_2,\Phi_1) = -\res(\Phi_1,\Phi_2) = -2.$$

\section{Basic properties of dynatomic polynomials for $x^d$}

In this section, we study the dynatomic polynomials of $\phi(x) = x^d$, where $d\geq 2$. We first recall the definition of the dynatomic polynomial of $\phi(x)$:
$$\Phi_{\phi,n}(x) = \prod_{k\mid n}(\phi^k(x)-x)^{\mu(n/k)} = \prod_{k\mid n}(x^{d^k}-x)^{\mu(n/k)}.$$
For convenience, we denote $\Phi_{\phi,n}$ by $\Phi_{d,n}$. We regard these polynomials as elements of $\zz[x]$, embedded in $K[x]$ when necessary.

We recall the following properties of Euler's totient function and the Möbius function. For every positive integer $n$, we have
$$\sum_{k\mid n}\varphi(k) = n,\qquad\sum_{k\mid n}\mu(k) = \begin{cases}
    1&\text{ if }n = 1,\\
    0&\text{ if }n>1.
\end{cases}$$
We also recall the following multiplicative version of the Möbius inversion formula. Let $f$ and $g$ be arithmetic functions with values in a commutative multiplicative group. Then, for every positive integer $n$, the following equivalence holds:
$$g(n) = \prod_{k\mid n}f(k)\Leftrightarrow f(n) = \prod_{k\mid n}g(k)^{\mu(n/k)}.$$
We refer the reader to \cite[Chapter 2]{Apostol76} for proofs of these results.

In Section 5.1 of \cite{Vivaldi92}, a decomposition of $\Phi_{2,n}$ is given for $n\geq 1$. In this section, we prove that $\Phi_{d,n}$ factors into a product of cyclotomic polynomials for all $d,n\geq 2$. For $k\geq 2$ relatively prime to $d$, let $\ord_k(d)$ denote the multiplicative order of $d$ modulo $k$. We use the convention $\ord_1(d) = 1$. Every index condition $\ord_k(d) = n$ below includes $\gcd(k,d) = 1$.

\begin{lemma}
    For every $d,n\geq 2$,
    $$\Phi_{d,n}(x) = \prod_{\ord_k(d) = n}\Phi_k(x).$$
\end{lemma}

\begin{proof}
    Note that
    $$\Phi_{d,n}(x) = \prod_{m\mid n}(x^{d^m}-x)^{\mu(n/m)},\qquad x^{d^m-1}-1 = \prod_{k\mid d^m-1}\Phi_k(x).$$
    Hence, we obtain
    $$\Phi_{d,n} = \prod_{m\mid n}\left(x\prod_{k\mid d^m-1}\Phi_k\right)^{\mu(n/m)}
        = x^{\sum_{m\mid n}\mu(n/m)}\prod_{m\mid n}\prod_{k\mid d^m-1}\Phi_k^{\mu(n/m)}
        = \prod_{m\mid n}\prod_{k\mid d^m-1}\Phi_k^{\mu(n/m)}$$
    since $\sum_{m\mid n}\mu(n/m) = 0$ when $n\geq 2$. 
    
    Fix an integer $n$ and a divisor $m$ of $n$. Then $k\mid d^m-1$ implies that $k\mid d^n-1$. Conversely, fix a divisor $k$ of $d^n-1$. We wish to determine the divisors $m$ of $n$ such that $k\mid d^m-1$. Note that $k\mid d^m-1$ if and only if $\ord_k(d)\mid m$. Letting $m = \ord_k(d)\cdot t$, the divisor $m$ divides $n$ if and only if $t$ divides $n/\ord_k(d)$. Hence, we obtain
    $$\Phi_{d,n} = \prod_{m\mid n}\prod_{k\mid d^m-1}\Phi_k^{\mu(n/m)} = \prod_{k\mid d^n-1}\prod_{t\mid \frac{n}{\ord_k(d)}}\Phi_k^{\mu(\frac{n}{\ord_k(d)\cdot t})}.$$
    Using the property of the Möbius function, we obtain
    $$\prod_{t\mid \frac{n}{\ord_k(d)}}\Phi_k^{\mu(\frac{n}{\ord_k(d)\cdot t})} = \begin{cases}
        \Phi_k &\text{ if }n = \ord_k(d)\\
        1&\text{ if }n\neq \ord_k(d).
    \end{cases}$$
    Since $\ord_k(d) = n$ implies that $k\mid d^n-1$, we conclude that $\Phi_{d,n} = \prod_{\ord_k(d) = n}\Phi_k$.
\end{proof}

\begin{remark}
    Morton and Patel \cite[Theorem 2.5]{Morton94} proved that if $\phi(x)\in K[x]$ with $\deg\phi\geq 2$, then $\prod_{k\mid n}(\phi^k(x)-x)^{\mu(n/k)}$ is a polynomial. In the special case $\phi(x) = x^d$, Lemma 3.1 yields an independent proof for $n\geq 2$ by expressing $\Phi_{d,n}$ as a product of cyclotomic polynomials. For $n = 1$, one has
    $$\Phi_{d,1}(x) = x(x^{d-1}-1) = x\prod_{k\mid d-1}\Phi_k(x) = x\prod_{\ord_k(d) = 1}\Phi_k(x).$$
\end{remark}

Since $\Phi_n(x) = \prod_{k\mid n}(x^k-1)^{\mu(n/k)}$, the dynatomic polynomials $\Phi_{d,n}$ may be viewed as dynamical analogues of cyclotomic polynomials. The following proposition gives dynamical analogues of well-known properties of cyclotomic polynomials.

\begin{proposition}For integers $d,n\geq 2$, the following hold.
    \begin{enumerate}
        \item $\deg\Phi_{d,n}(x) = \sum_{k\mid n}\mu(n/k)d^k = \sum_{\ord_k(d) = n}\varphi(k)$.
        \item $x^{d^n}-x = \prod_{k\mid n}\Phi_{d,k}(x)$.
        \item $\Phi_{d,n}(x)$ is a palindromic polynomial of even degree.
    \end{enumerate}
\end{proposition}

\begin{proof}~
    \begin{enumerate}
        \item By definition, $\Phi_{d,n}(x) = \prod_{k\mid n}(x^{d^k}-x)^{\mu(n/k)}$. Taking degrees gives the first equality. For the second equality, since the degree of the $n$-th cyclotomic polynomial is $\varphi(n)$ and $\Phi_{d,n}(x) = \prod_{\ord_k(d) = n}\Phi_k(x)$, we obtain $\deg \Phi_{d,n}(x) = \sum_{\ord_k(d) = n}\varphi(k)$.
        \item Since $\Phi_{d,n}(x) = \prod_{k\mid n}(x^{d^k}-x)^{\mu(n/k)}$, the result follows from the multiplicative version of the Möbius inversion formula.
        \item Note that $\deg\Phi_{d,n}(x) = \sum_{k\mid n}\mu(n/k)d^k$. If $d$ is even, it is clear that $\deg\Phi_{d,n}$ is even. If $d$ is odd, then
        $$\sum_{k\mid n}\mu(n/k)d^k\equiv \sum_{k\mid n}\mu(n/k) = 0\pmod 2.$$
        Therefore, the degree of $\Phi_{d,n}$ is even. Since $n\geq 2$, every index $k$ with $\ord_k(d) = n$ satisfies $k\geq 3$. Each such cyclotomic polynomial $\Phi_k$ is palindromic, so their product $\Phi_{d,n}(x)$ is also palindromic.
    \end{enumerate}
\end{proof}

For every $n\geq 1$, the cyclotomic polynomial $\Phi_n$ is irreducible over $\qq$. In contrast, for integers $d,n\geq 2$, the polynomial $\Phi_{d,n}$ is not always irreducible over $\qq$. We provide the following characterization. Before proving it, we recall the following theorem of Zsigmondy \cite{Zsigmondy92}.

\begin{theorem}
    Let $a>b>0$ be coprime integers, and let $n\geq 1$. Then, there exists a prime divisor of $a^n-b^n$ that does not divide $a^k-b^k$ for any $1\leq k<n$, with the following exceptions:
    \begin{enumerate}
        \item $n = 1$ and $a-b = 1$;
        \item $n = 2$ and $a+b$ is a power of $2$;
        \item $(a,b,n) = (2,1,6)$.
    \end{enumerate}
\end{theorem}

\begin{proposition}
    For integers $d,n\geq 2$, $\Phi_{d,n}(x)$ is irreducible over $\qq$ if and only if $d^n-1$ is a prime number. In that case $d = 2$ and $n$ is prime.
\end{proposition}

\begin{proof}
    Note that $\Phi_{d,n} = \prod_{\ord_k(d) = n}\Phi_k(x)$. Moreover, $\ord_{d^n-1}(d) = n$, since $0<d^i-1<d^n-1$ for $1\leq i<n$. If $d^n-1$ is a prime number, its only divisors are 1 and $d^n-1$, and hence $\Phi_{d,n}(x) = \Phi_{d^n-1}(x)$, which is irreducible over $\qq$. In particular, $d^n-1$ can be prime only if $d = 2$ and $n$ is a prime number, that is, $d^n-1$ is a Mersenne prime.

    Conversely, suppose that $d^n-1$ is not a prime number. If the triple $(d,1,n)$ is not one of the exceptional cases in Theorem 3.4, then there exists a prime $p\mid d^n-1$ such that $p\nmid d^k-1$ for all $1\leq k<n$. Thus, $\ord_p(d) = n$. Since $p<d^n-1$ and $\ord_{d^n-1}(d) = n$, the two distinct factors $\Phi_p$ and $\Phi_{d^n-1}$ divide $\Phi_{d,n}$.
    
    If $(d,1,n) = (2,1,6)$, then $\Phi_{2,6} = \Phi_{63}\cdot\Phi_{21}\cdot\Phi_9$, which is reducible over $\qq$. If $n = 2$ and $d+1 = 2^c$ for some $c\geq 2$, then $d-1 = 2^c-2$ and $d^2-1 = 2^{c+1}(2^{c-1}-1)$. Hence, $4\mid d^2-1$ and $4\nmid d-1$, which implies that $\ord_4(d) = 2$. Since $4<d^2-1$, the distinct factors $\Phi_4$ and $\Phi_{d^n-1}$ divide $\Phi_{d,n}$. This proves reducibility in all cases.
\end{proof}

In the remainder of this section, we denote the valuation of $n$ at $p$ by $v_p(n)$, where $p$ is prime and $n$ is a nonzero integer. We compute the values of $\Phi_{d,n}(x)$ at $x = 0$ and $x = 1$. Since $\Phi_{d,n}$ is a product of cyclotomic polynomials, we first recall the corresponding values for cyclotomic polynomials. For $n\geq 2$, it is well known that $\Phi_n(0) = 1$. On the other hand,
\begin{align*}
    \Phi_n(1) = \begin{cases}
        1&\text{ if }n\text{ is not a prime power},\\
        p&\text{ if }n = p^a\text{ for a prime }p\text{ and }a\geq 1.
    \end{cases}
\end{align*}
We refer the reader to \cite[Section 4.1]{Lang94} for a proof.

Note that $\Phi_{d,n}(1) = \prod_{\ord_k(d) = n}\Phi_k(1)$. Thus, for every prime divisor $p$ of $\Phi_{d,n}(1)$, computing $v_p(\Phi_{d,n}(1))$ amounts to determining the number of positive integers $a$ such that $\ord_{p^a}(d) = n$. 

For integers $d\geq 2$ and $n\geq 1$, and a prime $p\nmid d$, define
$$S(p,d,n)\coloneqq\{a\geq 1:\ord_{p^a}(d) = n\}.$$
For $p\mid d$, we set $S(p,d,n) = \emptyset$. We prove that, for $n\geq 2$, $v_p(\Phi_{d,n}(1)) = |S(p,d,n)| = v_p(\Phi_n(d))$.

\begin{lemma}
    Let $p$ be a prime number, and let $d,n\geq 2$ be integers. We have
    $$v_p(\Phi_n(d)) = |S(p,d,n)|.$$
\end{lemma}

\begin{proof}
    Suppose first that $p\nmid d$. For every $k\geq 1$, the condition $p^a\mid d^k-1$ is equivalent to $\ord_{p^a}(d)\mid k$. Hence
    $$v_p(d^k-1) = \sum_{r\mid k}|S(p,d,r)|.$$
    Applying additive Möbius inversion gives
    $$|S(p,d,n)| = \sum_{r\mid n}\mu(n/r)v_p(d^r-1) = v_p\left(\prod_{r\mid n}(d^r-1)^{\mu(n/r)}\right) = v_p(\Phi_n(d)).$$
    If $p\mid d$, then $p$ does not divide the integer $d^n-1$, so the valuation $v_p(\Phi_n(d))$ is zero, in agreement with $S(p,d,n) = \emptyset$.
\end{proof}

\begin{corollary}
    For integers $d,n\geq 2$, we have $\Phi_{d,n}(0) = 1$ and $\Phi_{d,n}(1) = \Phi_n(d)$.
\end{corollary}

\begin{proof}
    By Lemma 3.1, we have $\Phi_{d,n}(x) = \prod_{\ord_k(d) = n}\Phi_k(x)$. Since $n\geq 2$, every index $k$ in this product is at least 3. So evaluation at 0 gives $\Phi_{d,n}(0) = \prod_{\ord_k(d) = n}\Phi_k(0) = 1$. 

    To evaluate $\Phi_{d,n}(x)$ at $x = 1$, note that only factors with prime power indices contribute to a prime valuation. Therefore, for every prime $p$, Lemma 3.6 gives
    $$v_p(\Phi_{d,n}(1)) = \sum_{\ord_{p^a}(d) = n}v_p(\Phi_{p^a}(1)) = |S(p,d,n)| = v_p(\Phi_n(d)).$$
    When $p\mid d$, the summation is empty and both valuations are zero. Since both values $\Phi_{d,n}(1)$ and $\Phi_n(d)$ are positive integers, the equality of all valuations proves the result.
\end{proof}

\section{The resultants of dynatomic polynomials for $x^d$}

Let $1\leq n<m$ be fixed. By Proposition 3.3(3), the degree of $\Phi_{d,m}$ is even. Thus,
$$\res(\Phi_{d,m},\Phi_{d,n}) = (-1)^{\deg \Phi_{d,m}\deg \Phi_{d,n}}\res(\Phi_{d,n},\Phi_{d,m}) = \res(\Phi_{d,n},\Phi_{d,m}).$$
Therefore, it suffices to compute $\res(\Phi_{d,n},\Phi_{d,m})$ for $1\leq n<m$. 

For $n>1$, Lemma 3.1 and Proposition 2.1(3) give
\begin{equation}
    \res(\Phi_{d,n},\Phi_{d,m}) = \prod_{\ord_a(d) = n}\prod_{\ord_b(d) = m}\res(\Phi_a,\Phi_b).
\end{equation}
The identity (4.1) also holds for $n = 1$: the additional factor $x$ in $\Phi_{d,1}$ contributes $\res(x,\Phi_{d,m}) = \Phi_{d,m}(0) = 1$ by Corollary 3.7. Since $n\neq m$, we have $a\neq b$. Moreover $m>1$ implies $b\geq 3$, so the sign exception following Theorem 2.2 cannot occur. Hence each factor in (4.1) is positive, and the resultant is a positive integer. 

To determine when this resultant exceeds 1, recall that $\res(\Phi_a,\Phi_b)>1$ only when either $a/b$ or $b/a$ is a positive power of a prime. Thus, we first investigate when a divisibility relation between $a$ and $b$ occurs.

\begin{lemma}
    Let $d,m,n$ be positive integers with $d\geq 1$ and $n\nmid m$. If $a$ and $b$ are positive integers satisfying
    $$\ord_a(d) = n,\quad \ord_b(d) = m,$$
    then $a\nmid b$.
\end{lemma}

\begin{proof}
    Write $m = nq+r$, where $q,r\in\zz$ and $0<r<n$. If $a\mid b$, then $a\mid d^m-1$. Hence,
    $$1\equiv d^m\equiv d^{nq+r}\equiv d^r\pmod a.$$
    Thus $\ord_a(d)\leq r<n$, contradicting $\ord_a(d) = n$.
\end{proof}

The following corollary is an immediate consequence of Lemma 4.1.

\begin{corollary}
    Let $d\geq 2$ and $n<m$ be positive integers. If $\ord_a(d) = n$ and $\ord_b(d) = m$, then $b\nmid a$. 
\end{corollary}

\begin{theorem}
    If $d\geq 2$ and $m>1$, then $\res(\Phi_{d,1},\Phi_{d,m}) = \Phi_m(d)^{d-1}$.
\end{theorem}

\begin{proof}
    Recall that 
    $$\Phi_{d,1}(x)= x\prod_{\ord_a(d) = 1}\Phi_a(x),\qquad\Phi_{d,m}(x) = \prod_{\ord_b(d) = m}\Phi_b(x).$$
    By Proposition 2.1(1), we have $\res(x,\Phi_{d,m}) = \Phi_{d,m}(0) = 1$. Hence,
    $$\res(\Phi_{d,1},\Phi_{d,m}) = \prod_{\ord_a(d) = 1}\res(\Phi_a,\Phi_{d,m}) = \prod_{\ord_a(d) = 1}\prod_{\ord_b(d) = m}\res(\Phi_a,\Phi_b).$$
    Since $m>1$, we have $b\nmid a$ by Corollary 4.2. Thus, a nontrivial factor can occur only when $b/a$ is a prime power.
    
    Fix a divisor $a$ of $d-1$ and a prime number $p\nmid d$. Write $a = a'p^k$, where $k = v_p(a)$ and $\gcd(a',p) = 1$. Since $\ord_{a'}(d) = 1$, the Chinese remainder theorem gives
    $$\ord_{ap^s}(d) = \ord_{a'\cdot p^{k+s}}(d) = \lcm(\ord_{a'}(d),\ord_{p^{k+s}}(d)) = \ord_{p^{k+s}}(d).$$
    By Theorem 2.2, $\res(\Phi_a,\Phi_b) = p^{\varphi(a)}$ if and only if $b = p^s\cdot a$ for some $s\geq 1$. Moreover, such an integer $s$ must satisfy $\ord_b(d) = \ord_{p^{k+s}}(d) = m$. Since $p^k\mid a$ and $a$ is a divisor of $d-1$, the order of $d$ modulo $p^t$ is 1 for $1\leq t\leq k$. As $m>1$, the map
    \begin{align*}
        \{s\geq 1:\ord_{p^{k+s}}(d) = m\} &\to  S(p,d,m)\\
        s&\mapsto k+s
    \end{align*}
    is a bijection. Consequently,
    \begin{align*}
        v_p(\res(\Phi_{d,1},\Phi_{d,m})) &= \sum_{\ord_a(d) = 1}\varphi(a)\cdot|\{s\geq 1:\ord_{p^{k+s}}(d) = m\}|\\
        &= \sum_{a\mid d-1}\varphi(a)\cdot|S(p,d,m)|\\
        &= \sum_{a\mid d-1}\varphi(a)\cdot v_p(\Phi_m(d))\\
        &= (d-1)v_p(\Phi_m(d)).
    \end{align*}
    If $p\mid d$, there is no index $b = ap^s$ relatively prime to $d$, so $v_p(\res(\Phi_{d,1},\Phi_{d,m})) = 0$. On the other hand, $\Phi_m(d) \equiv \Phi_m(0) = 1\pmod p$. The two positive integers thus have the same valuation at every prime, proving the formula.
\end{proof}

Finally, we investigate the resultant $\res(\Phi_{d,n},\Phi_{d,m})$ for $1<n<m$. 

\begin{theorem}
    Let $d\geq 2$ and $1<n<m$. The following statements hold. 
    \begin{enumerate}
        \item $\res(\Phi_{d,n},\Phi_{d,m}) = 1$ if and only if $n\nmid m$.
        \item If $n\mid m$, then
        $$\res(\Phi_{d,n},\Phi_{d,m}) = \Phi_{m/n}(d^n)^{\deg \Phi_{d,n}}.$$
    \end{enumerate}
\end{theorem}

\begin{proof}
    \begin{enumerate}
        \item Suppose first that $n\nmid m$. Lemma 4.1 and Corollary 4.2 give $a\nmid b$ and $b\nmid a$ whenever $\ord_a(d) = n$ and $\ord_b(d) = m$. Therefore, every resultant $\res(\Phi_a,\Phi_b)$ is 1, which implies that $\res(\Phi_{d,n},\Phi_{d,m}) = 1$.

        Conversely, if $n\mid m$, part (2), proved below, gives a positive power of $\Phi_{m/n}(d^n)$. This cyclotomic value is greater than 1. Indeed, for integers $D,k\geq 2$, $\Phi_2(D) = D+1>1$; for $k>2$, primitive $k$-th roots of unity $\zeta$ occur in complex-conjugate pairs, and each pair contributes
        $$(D-\zeta)(D-\overline{\zeta}) = D^2-2D\re(\zeta)+1>(D-1)^2\geq 1.$$
        Also $\deg \Phi_{d,n}>0$, since $\ord_{d^n-1}(d) = n$ and thus $\Phi_{d^n-1}$ is a factor of $\Phi_{d,n}$. Hence the resultant exceeds 1.
        \item Suppose $n\mid m$, and put $h = m/n>1$. For a prime $p\nmid d$, the identity (4.1) gives
        \begin{equation}
            v_p(\res(\Phi_{d,n},\Phi_{d,m})) = \sum_{\ord_a(d) = n}\varphi(a)\cdot|T(a,p,d,m)|,
        \end{equation}
        where 
        $$T(a,p,d,m) = \{t\geq 1:\ord_{ap^t}(d) = m\}.$$
        We claim that for every $a$ with $\ord_a(d) = n$,
        \begin{equation}
            |T(a,p,d,m)| = |S(p,d^n,h)| = v_p(\Phi_{h}(d^n)).
        \end{equation}
        Write $a = a'p^k$, where $k = v_p(a)$ and $\gcd(a',p) = 1$. For $s\geq k$, the Chinese remainder theorem gives 
        $$\ord_{a'p^s}(d) = \lcm(\ord_{a'}(d),\ord_{p^s}(d)).$$
        Since $\ord_{p^k}(d)\mid \ord_{p^s}(d)$ and 
        $$n = \ord_{a}(d) = \lcm(\ord_{a'}(d),\ord_{p^k}(d)),$$
        it follows that
        $$\ord_{a'p^s}(d) = \lcm(n,\ord_{p^s}(d)).$$
        Here $\ord_{p^0}(d) = \ord_1(d) = 1$ by convention.

        Recall that if an element $g$ of a group has order $r$, then $g^n$ has order $r/\gcd(r,n)$. Applying this fact modulo $p^s$ gives
        $$\ord_{p^s}(d^n) = \frac{\ord_{p^s}(d)}{\gcd(n,\ord_{p^s}(d))} = \frac{\lcm(n,\ord_{p^s}(d))}{n} = \frac{\ord_{a'p^s}(d)}{n}.$$
        Consequently, for $t\geq 1$,
        $$\ord_{ap^t}(d) = m\Leftrightarrow \ord_{p^{k+t}}(d^n) = h.$$
        Moreover, $p^k\mid a$ and $a$ is a divisor of $d^n-1$, so the order of $d^n$ modulo $p^s$ is 1 for $1\leq s\leq k$. Since $h>1$, every element of $S(p,d^n,h)$ is greater than $k$. Thus 
        \begin{align*}
            T(a,p,d,m)&\to S(p,d^n,h)\\
            t&\mapsto k+t
        \end{align*}
        is a bijection. Applying Lemma 3.6 with base $d^n$, we prove the equation (4.3).

        Substituting into (4.2) and using Proposition 3.3(1), we obtain
        $$v_p(\res(\Phi_{d,n},\Phi_{d,m})) = \sum_{\ord_a(d) = n}\varphi(a)v_p(\Phi_h(d^n)) = \deg \Phi_{d,n}\cdot v_p(\Phi_h(d^n)).$$
        If $p\mid d$, both sides vanish, exactly as in the proof of Theorem 4.3. Since both integers in part (2) are positive, equality of all prime valuations proves the formula.
    \end{enumerate}
\end{proof}

Recall that if $\alpha$ is a periodic point of a polynomial $\phi(x)$, then having exact period $n$ implies having formal period $n$. If a point of formal period $n$ has exact period $m<n$, then $m\mid n$, and it is also a root of $\Phi_{\phi,m}$. Thus, $\res(\Phi_{\phi,m},\Phi_{\phi,n}) = 0$. By Theorems 4.3 and 4.4, the resultant $\res(\Phi_{d,m},\Phi_{d,n})\neq 0$ for all distinct positive integers $m,n$. Therefore, we obtain the following corollary.

\begin{corollary}
    Over a field of characteristic zero, if $\alpha$ has formal period $n$ for $\phi(x) = x^d$, then $\alpha$ has exact period $n$.
\end{corollary}

\begin{remark}
    The exact formula can also be recovered from the multiplier polynomial relation recalled in the introduction. Indeed, for $n>1$, Lemma 3.1 shows that $\Phi_{d,n}$ is squarefree over a field of characteristic zero. Its roots are nonzero, and 
    $$(\phi^n)'(\alpha) = d^n\alpha^{d^n-1} = d^n.$$
    By Corollary 4.5, these roots form disjoint cycles of length $n$. Thus $n\mid \deg\Phi_{d,n}$ and
    $$\delta_n(x) = (x-d^n)^{\deg\Phi_{d,n}/n}$$
    for $n>1$. For $n = 1$, the fixed point $0$ has multiplier $0$, and the other $d-1$ fixed points have multiplier $d$, so
    $$\delta_1(x) = x(x-d)^{d-1}.$$
    Substituting these expressions into the cited relation gives the same resultant formulas when $n\mid m$, up to sign; positivity was established at the beginning of this section. The preceding proofs derive the exact values directly from cyclotomic resultants.
\end{remark}

\bibliographystyle{plain}
\bibliography{Reference}

\end{document}